\documentclass{article}
\usepackage{amsmath,amsfonts,amssymb,graphicx,bbm,bm,amsthm}
\usepackage{tikz}
\usepackage[round]{natbib}

\usetikzlibrary{calc}

\newcommand{\X}{\mathcal{X}}
\newcommand{\Y}{\mathcal{Y}}
\newcommand{\Z}{\mathcal{Z}}

\newcommand{\G}{\mathcal{G}}

\newcommand{\bs}{\boldsymbol}
\newcommand{\A}{{\bs{A}}}
\renewcommand{\a}{{\bs{a}}}
\newcommand{\B}{{\bs{B}}}
\renewcommand{\b}{{\bs{b}}}
\renewcommand{\c}{{\bs{c}}}
\newcommand{\C}{{\bs{C}}}
\renewcommand{\d}{{\bs{d}}}
\newcommand{\D}{{\bs{D}}}
\newcommand{\V}{{\bs{V}}}
\newcommand{\U}{{\bs{U}}}
\newcommand{\W}{{\bs{W}}}
\DeclareMathOperator{\ACDE}{ACDE}
\DeclareMathOperator{\Do}{do}
\DeclareMathOperator{\pa}{pa}

\newcommand\indep{\protect\mathpalette{\protect\independenT}{\perp}}
\def\independenT#1#2{\mathrel{\rlap{$#1#2$}\mkern2mu{#1#2}}}

 \newtheorem{lem}{Lemma}[section]
 \newtheorem{thm}[lem]{Theorem}
 \newtheorem{prop}[lem]{Proposition}
 \newtheorem{cor}[lem]{Corollary}

 \newtheorem{rmk}[lem]{Remark}

\title{Graphical Methods for Inequality Constraints in Marginalized DAGs}
\author{Robin J.~Evans}
\begin{document}

\maketitle

\begin{abstract}
We present a graphical approach to deriving inequality constraints for directed acyclic graph (DAG) models, where some variables are unobserved.  
In particular we show that the observed distribution of a discrete model is always restricted if any two observed variables are neither adjacent in the graph, nor share a latent parent; this generalizes the well known instrumental inequality.  
The method also provides inequalities on interventional distributions, which can be used to bound causal effects.  
All these constraints are characterized in terms of a new graphical separation criterion, providing an easy and intuitive method for their derivation.
\end{abstract}

% \begin{keywords}
% Causal model, controlled direct effect, directed acyclic graph, intervention, marginalized.
% \end{keywords}
%
\section{Introduction}
\label{sec:intro}

Models based on directed acyclic graphs (DAGs) are commonly used for causal inference on account of their simple to understand conditional independence constraints, and the intuitive appeal of using arrows to display causal dependences.  If all the variables in a DAG are observed then causal quantities of interest are typically point identified, and derivable in terms of conditional probabilities.  
However, it is common for some variables to be unobservable, possibly representing confounding factors which may bias inference;
in this case we can only observe the marginal distribution over the remaining variables.

%; such systems may be represented in terms of a DAG model in which some of the variables have been marginalized over.

The models which result from the marginalization of a DAG are much less well understood and, unlike DAGs, are not described merely in terms of conditional independence constraints.  In particular, causal effects may not be point identified, and we can only hope for inequality constraints describing the range of possible values.  %If such bounds exclude the possibility of no effect, then this can be used to falsify a simpler model.

Existing methods for deriving bounds on observed distributions are either specific to a particular model \cite{pearl:95, balke:97}, or computationally intensive and lacking the intuitiveness of a graphical approach \cite{bonet:01, kang:06}.
%Falsifiable bounds on the observed distribution of the instrumental variables (IV) model, shown in Figure \ref{fig:iv}(a), were first derived by \cite{pearl:95}, and further work was done by \cite{bonet:01} using linear programming (LP).  Causal bounds for the same model were derived by \cite{balke:97} and \cite{cai:08} using LP.  
%\cite{kang:06} give general methods for deriving bounds on interventional distributions, although the algorithms are fairly computationally intensive.  
%\cite{ramsahai:12} derives bounds in the decision theoretic causal framework.
See \cite{ramsahai:12} for an approach which is graphical in spirit, but uses computationally difficult variable elimination methods.
In this paper we take steps to remedy these problems by providing a simple graphical separation criterion for determining the existence of constraints, and for constructing them explicitly.

The remainder of the paper is organised as follows: \S\ref{sec:dags} introduces DAGs and related terminology and notation.  \S\ref{sec:iv} gives a new method for deriving known constraints on the observed distribution of the instrumental variables model, and related causal effects.  \S\ref{sec:other} applies these methods to give new constraints for general DAG models, and \S\ref{sec:exm} contains examples.  A discussion is found in \S\ref{sec:discuss}, and longer proofs are in an appendix.

\section{Graphical Models}
\label{sec:dags}

A \emph{directed graph} $\G$ is a set of vertices $\V$, with a collection of ordered pairs of distinct vertices, or \emph{edges}, $\mathcal{E}$.  If $(X,Y) \in \mathcal{E}$ we write $X \rightarrow Y$, and say that $X$ is a \emph{parent} of $Y$.  The set of parents of $Y$ is denoted $\pa_\G(Y)$.
A \emph{path} is a sequence of adjacent edges in a graph, without repetition of vertices; for example, the graph in Figure \ref{fig:iv}(a) contains the path $\pi_1 : Z \rightarrow X \leftarrow U \rightarrow Y$.  A path is \emph{directed} from $X$ to $Y$ if all the arrows point away from $X$ and towards $Y$.  If there is a directed path from $X$ to $Y$ we say that $Y$ is a \emph{descendant} of $X$, and $X$ an \emph{ancestor} of $Y$.
A directed graph is \emph{acyclic} if there is no directed path from a vertex to itself; such an object is called a directed acyclic graph (DAG).

%A \emph{directed acyclic graph} $\G$ is a collection of vertices $\V$ under a total ordering $\prec$, and a collection of unordered pairs of vertices, or \emph{edges}, $\mathcal{E}$.  If $\{X,Y\} \in E$ and $X \prec Y$ we write $X \rightarrow Y$, and say that $X$ is a \emph{parent} of $Y$.  The set of parents of $Y$ is denoted $\pa_\G(Y)$.

We associate each vertex $X$ with a random variable under some multivariate distribution $P$; let $P$ admit a density $f$.  
For convenience, in what follows we will use $X$ to denote both the vertex and the random variable, and similarly use operators and bold face letters (e.g.\ $\pa_\G(X)$, $\C$) to refer to both a set of vertices and the associated vector of random variables.
%; this slight abuse of notation is unambiguous in practice, and fits with the approach of other authors.
The \emph{factorization criterion} for DAGs says that $P$ is in the model corresponding to the DAG $\G$ if the joint density factorizes as $\prod_{V \in \V} f(V \,|\, \pa_{\G}(V))$.
%\begin{align}
%\prod_{V \in \V} f(V \,|\, \pa_{\G}(V)). \label{eqn:factor}
%\end{align}
%We denote the set of probability measures which factorize according to a DAG $\G$ by $\M(\G)$.

Internal vertices on a path with two adjacent arrowheads are called \emph{colliders} on the path; other internal vertices are non-colliders.  On the path $\pi_1$, $X$ is a collider, and $U$ a non-collider.
A path $\pi$ from $X$ to $Y$ is \emph{blocked} given a set of vertices $\C$ if there is a non-collider on $\pi$ in $\C$, or a collider on $\pi$ which is not an ancestor of any vertex in $\C$.

We say that two sets of vertices $\A$ and $\B$ are \emph{d-separated} given a set of vertices $\C$, if every path from any vertex in $\A$ to any vertex in $\B$ is blocked by $\C$.  A probability distribution $P$ obeys the \emph{global Markov property} for a DAG $\G$ if whenever $\A$ and $\B$ are d-separated by $\C$ in $\G$, then $\A \indep \B \,|\, \C \, [P]$.

It is well known that d-separation is equivalent to the factorization criterion \cite{verma:88}.  In particular, all constraints implied by a DAG on fully observed random variables can be interpreted as conditional independences.

Assigning a causal interpretation to a DAG model requires extra assumptions, in particular that the system under observation is stable under interventions with respect to the graph.  We will denote an intervention to fix $X = x$ by $\Do(X = x)$, or $\Do(x)$ for short; graphically this may be represented by removing the edges of the form $Y \rightarrow X$, so that $X$ has no parents in the new graph.  The density $f( \V \,|\, \Do(x))$ is given by dividing the joint density $f$ by $f(x \,|\, \pa_{\G}(X))$ and multiplying by the indicator function $\mathbbm{1}_{\{X = x\}}$.  See \cite{pearl:09} for details.

If some of the variables in a DAG are unobserved, we may be interested in the implications of the underlying graph for the observable margin.  Let $\U \subset \V$ denote the set of latent or unobservable vertices; the observable margin is then
\begin{align}
\int_{U \in \U}  \prod_{V \in \V} f(V \,|\, \pa_{\G}(V)) \, dU. \label{eqn:factor}
\end{align}
%We will use $V$ to denote a generic variable, $U$ for an unobserved variable, $X$, $Y$, $Z$, $W$ for observed variables.  
%
The marginal distribution over the observed variables is completely identifiable, but some of the structure of the underlying graph may be impossible to determine in the presence of latent variables.  We will make no assumption about the state space of the latent variables, since these are unobserved.  Some conditional independences may still be observable, but other kinds of constraint also arise, including \emph{Verma constraints} \cite{verma:91}, and inequalities on the observed distribution (see next section).  

%The observed vertices of a graph may be partitioned into \emph{districts}; $X$ and $Y$ lie in the same district if there is a path between $X$ and $Y$ on which no two adjacent vertices are both observable.  The graph in Figure \ref{fig:iv} has two districts, $\{X,Y\}$ and $\{Z\}$.  \cite{tian:02} use the equivalent term \emph{c-components}.

%Latent variables and their incident edges will be drawn in red (see Figure \ref{fig:iv}); this makes it easy to identify districts.  
Without loss of generality we will assume that none of the latent variables have any parents.

\section{Instrumental Variables} \label{sec:iv}

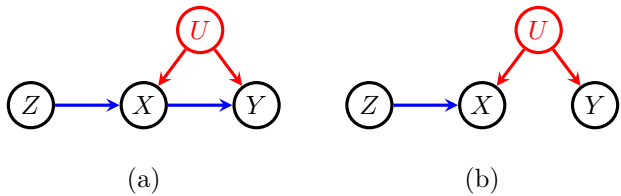
\begin{figure}
\begin{center}
\begin{tikzpicture}
[rv/.style={circle, draw, very thick, minimum size=6mm, inner sep=0.5mm}, node distance=15mm, >=stealth]
 \pgfsetarrows{latex-latex};
\begin{scope}
 \node[rv]  (1)              {$Z$};
 \node[rv, right of=1] (2) {$X$};
 \node[rv, above of=2, color=red, xshift=7.5mm, yshift=-5mm] (U) {$U$};
 \node[rv, right of=2] (3) {$Y$};
 \draw[->, very thick, color=blue] (1) -- (2);
 \draw[->, very thick, color=blue] (2) -- (3);
 \draw[->, very thick, color=red] (U) -- (2);
 \draw[->, very thick, color=red] (U) -- (3);
\node[below of=2, yshift=5mm] {(a)};
\end{scope}
\begin{scope}[xshift=45mm]
 \node[rv]  (1)              {$Z$};
 \node[rv, right of=1] (2) {$X$};
 \node[rv, above of=2, color=red, xshift=7.5mm, yshift=-5mm] (U) {$U$};
 \node[rv, right of=2] (3) {$Y$};
 \draw[->, very thick, color=blue] (1) -- (2);
% \draw[->, very thick, color=blue] (2) -- (3);
% \draw[->, very thick, dashed, color=blue] (1.320) .. controls +(1,-0.7) and +(-1,-0.7) .. (3.220);
 \draw[->, very thick, color=red] (U) -- (2);
 \draw[->, very thick, color=red] (U) -- (3);
\node[below of=2, yshift=5mm] {(b)};
\end{scope}
\end{tikzpicture}
\end{center}
\caption{(a) The instrumental variables (IV) model; $U$ is unobserved. (b) The IV model with the effect of $X$ on $Y$ removed.}
\label{fig:iv}
\end{figure}

Perhaps the most thoroughly studied causal DAG model is the \emph{instrumental variables} model, represented in Figure \ref{fig:iv}(a).  It arises naturally in randomized trials with imperfect compliance, in which $Z$ represents a randomized treatment assignment, $X$ the treatment actually taken by the subject, and $Y$ an outcome; $U$ represents unmeasured confounding factors which may affect both the probability of the subject taking the treatment and the outcome of interest, so that na\"ive estimators of the effect of $X$ on $Y$ will be biased.

The graph encodes (amongst other assumptions) that the assignment $Z$ does not affect the outcome $Y$ other than through the treatment $X$.  This is known as the exclusion restriction, and is important for assessing the effect of $X$ on $Y$; implications of the exclusion restriction which can be subjected to an empirical test are therefore very useful.

%The graph encodes two important assumptions: that the assignment $Z$ does not affect the outcome $Y$ other than through the treatment $X$, and that the assignment is independent of $U$.  Randomization of $Z$ makes the second assumption particularly plausible.  The first assumption, known as the exclusion restriction, is generally more difficult to assess even if it is scientifically plausible, so it is useful to find implications which can be subjected to an empirical test.

Making no assumptions about the character of $U$, and if $X$ is continuous, the observable margin is unconstrained \cite{bonet:01}.  However, if the observed variables have finite and discrete state spaces, then the observed distribution obeys the \emph{instrumental inequality} of \cite{pearl:95}:
\begin{align}
\max_{x} \sum_y \max_z p(x, y \,|\, z) \leq 1; \label{eqn:ins}
\end{align}
here $p(x,y \,|\, z)$ is used to denote $P(X=x, Y=y \,|\, Z=z)$.
This restriction can be used to falsify the IV model.  Pearl's proof of the inequality is model specific, and it is not clear how it might be applied to other graphs.  Below we present a new approach to the derivation of (\ref{eqn:ins}), and a more graphical interpretation of its meaning; as we shall see, this method can be adapted to many other DAG models, and provides some causal constraints.

\begin{prop}
Let $P$ be a probability distribution over three random variables $Z$, $X$ and $Y$, taking values in discrete sets $\mathcal{Z}$, $\mathcal{X}$ and $\mathcal{Y}$ respectively.
Then $P$ obeys the IV model only if for each $\xi \in \X$, the collection of conditional probabilities $(p(\xi, y \,|\, z), y \in \Y, z \in \Z)$ is \emph{compatible} with a distribution under which $Y \indep Z$.  

In other words, only if for each $\xi \in \X$ there exists a distribution $P^*$ such that $Y \indep Z \, [P^*]$, and $p^*(\xi, y \,|\, z) = p(\xi, y \,|\, z)$ for each $y \in \Y$ and $z \in \Z$.

This condition implies the instrumental inequality (\ref{eqn:ins}).
\end{prop}

\begin{proof}
Suppose that $P$ is in the IV model.  Then
\begin{align*}
p(x, y\,|\, z) = \int f(u)\, p(x \,|\, u, z) \, p(y \,|\, u,x) \, du;
\end{align*}
construct a distribution $P^*$ by
\begin{align*}
p^*(x, y\,|\, z) = \int f(u)\, p(x \,|\, u, z) \, p(y \,|\, u, \xi) \, du.
\end{align*}
Under $P^*$, the effect of $X$ on $Y$ has been broken, because $Y$ behaves as though $X=\xi$ regardless of its actual value.    $P^*$ obeys the factorization criterion with respect to the graph in Figure \ref{fig:iv}(b); thus $Y \indep Z \, [P^*]$, and by construction $p^*(\xi, y \,|\, z) = p(\xi, y \,|\, z)$ for each $y \in \Y$ and $z \in \Z$.

To see that this implies (\ref{eqn:ins}), first note that the independence is equivalent to
\begin{align*}
p^*(y \,|\, z) &= p^*(y \,|\, z')\\
p(\xi, \, y \,|\, z)\! +\! \sum_{x\neq \xi} p^*(x, \, y \,|\, z)\! &=\! p(\xi, \, y \,|\, z')\! +\! \sum_{x\neq \xi} p^*(x, \, y \,|\, z')
\end{align*}
for each $y \in \Y$, $z, z' \in \Z$.  Suppose we are given the probabilities $p(\xi, \, y \,|\, z)$ and asked to construct a distribution $P^*$ satisfying these equations.  Since all the quantities are positive, 
%\begin{align*}
%p(\xi, \, y \,|\, z') - p(\xi, \, y \,|\, z) \leq \sum_{x\neq \xi} p^*(x, \, y \,|\, z),
%\end{align*}
and this equality holds for each $z'$, we have
\begin{align*}
\max_{z'} p(\xi, \, y \,|\, z') - p(\xi, \, y \,|\, z) \leq \sum_{x\neq \xi} p^*(x, \, y \,|\, z).
\end{align*}
However the sum of the quantities on the RHS over $y$ cannot be greater than $1-p(\xi \,|\, z) = 1-\sum_y p(\xi, y \,|\, z)$, so 
\begin{align*}
\sum_y \left(\max_{z'} p(\xi, \, y \,|\, z') - p(\xi, \, y \,|\, z) \right) & \leq 1-\sum_y p(\xi, y \,|\, z)\\
\sum_y \max_{z'} p(\xi, \, y \,|\, z') &\leq 1.
\end{align*}
Applying this to each $\xi$ gives (\ref{eqn:ins}).
%It is easy to see that (\ref{eqn:ins}) implies the condition given in the Theorem, 
\end{proof}

\begin{rmk}
Whilst these inequalities are not new, the importance of the above result lies in the proof technique; we will see in the next section that it generalizes to many other DAG models, giving novel results.  

The instrumental inequality is exact when $X$, $Y$ and $Z$ are binary, but insufficient if $Z$ takes three states \cite{bonet:01}.  The sufficient bounds are difficult to derive without using computationally intensive linear programming techniques and Fourier-Motzkin elimination, which become infeasible for moderately sized state spaces.  
%These inequalities may be interpreted in terms of the above proof as saying that we must be able to find distributions $P^*$ for each $\xi \in \X$, as stated, but also such that in each case the factor $p^*(x \,|\, z, u)$ is the same.
\end{rmk}

\subsection{Causal bounds on the IV model}

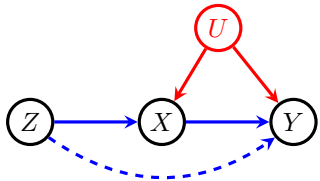
\begin{figure}
\begin{center}
\begin{tikzpicture}
[rv/.style={circle, draw, very thick, minimum size=6mm, inner sep=0.5mm}, node distance=17.5mm, >=stealth]
 \pgfsetarrows{latex-latex};
 \node[rv]  (1)              {$Z$};
 \node[rv, right of=1] (2) {$X$};
 \node[rv, above of=2, color=red, xshift=7.5mm, yshift=-5mm] (U) {$U$};
 \node[rv, right of=2] (3) {$Y$};
 \draw[->, very thick, color=blue] (1) -- (2);
 \draw[->, very thick, color=blue] (2) -- (3);
 \draw[->, very thick, dashed, color=blue] (1.320) .. controls +(1,-0.7) and +(-1,-0.7) .. (3.220);
 \draw[->, very thick, color=red] (U) -- (2);
 \draw[->, very thick, color=red] (U) -- (3);
\end{tikzpicture}
\end{center}
\caption{The model from Figure \ref{fig:iv}(a) with a possible effect of $Z$ on $Y$ added.}
\label{fig:iv2}
\end{figure}

%If we falsify the model in Figure \ref{fig:iv}(a) using the bounds (\ref{eqn:ins}) then we must conclude that $Z$ has a non-zero effect upon $Y$ (at least if we accept the other assumptions of the causal IV model).
We next try to invert the problem and ask \emph{how much} effect $Z$ can have on $Y$ given the observed distribution.  In some sense we are trying to quantify the strength of the dashed arrow in Figure \ref{fig:iv2}.  A suitable measure is the \emph{average controlled direct effect} (ACDE) of $Z$ on $Y$, controlling for $X=x$; this is defined for binary $Z$, $X$ and $Y$ as
%\begin{align*}
%\begin{split}
%\ACDE_{Z \rightarrow Y}(x) &\equiv p(y=1 \,|\, \Do(z=1, x))\\ & \qquad - p(y=1 \,|\, \Do(z=0, x)).
%\end{split}
%\end{align*}
\begin{align*}
\ACDE_{Z \rightarrow Y}(x) &\equiv p(y_1 \,|\, \Do(z_1, x)) - p(y_1 \,|\, \Do(z_0, x)).
\end{align*}
Here $y_1$ is a shorthand for $\{Y = 1\}$, whilst $x$ means $\{X=x\}$, etc.  
Generalizations to non-binary state spaces are also possible \cite{cai:08}.
Note that $\ACDE_{Z \rightarrow Y}(x) = 0$ for each $x$ if $Z \not\rightarrow Y$.
For the DAG in Figure \ref{fig:iv2},
\begin{align*}
p(y \,|\, \Do(z, x)) &= \int_u f(u) \, p(y \,|\, x, z, u) \, du,
%&\neq p(y \,|\, z, x),
\end{align*}
which is not identified.
However, constructing $P^*$ as above,
\begin{align*}
p(y \,|\, \Do(z, \xi)) &= \int_u f(u) \, p^*(y \,|\, z, u) \, du\\
&= p^*(y \,|\, z)\\
&= p(y, \xi \,|\, z) + \sum_{x \neq \xi} p^*(y,x \,|\, z)\\
%&= p(y, \xi \,|\, z) + \sum_{x \neq \xi} p^*(y,x \,|\, z)\\
&\leq p(y, \xi \,|\, z) + 1 - p(\xi \,|\, z).
\end{align*}
Also $p(y \,|\, \Do(z, \xi)) \geq p(y, \xi \,|\, z)$, so
%\begin{align*}
%\ACDE(x) &\leq p(y=1, x \,|\, z=1) + 1 - p(x \,|\, z=1) - p(y=1, x \,|\, z=0)\\
%&= 1 - p(y=0, x \,|\, z=1) - p(y=1, x \,|\, z=0),
%\end{align*}
\begin{align*}
\ACDE_{Z \rightarrow Y}(x) &\leq p(y_1, x \,|\, z_1) \!+\! 1 \!-\! p(x \,|\, z_1) \!-\! p(y_1, x \,|\, z_0)\\
&= 1 - p(y_0, x \,|\, z_1) - p(y_1, x \,|\, z_0),
\end{align*}
and similarly
\begin{align*}
\ACDE_{Z \rightarrow Y}(x) &\geq p(y_1, x \,|\, z_1) + p(y_0, x \,|\, z_0) - 1.
\end{align*}
Note that the ACDE bounds include zero if and only if the instrumental inequality (\ref{eqn:ins}) is satisfied.  These bounds were derived by \cite{cai:08} using linear programming, and shown to be tight.  In the next section we will extend this method to other graphs.

\section{Other Models} \label{sec:other}

Just as d-separation provides a graphical criterion for finding observable conditional independences, we now provide a graphical criterion for finding observable inequality constraints.  For a DAG $\G$ with vertex set $\V$ and edge set $\mathcal{E}$, define the \emph{induced subgraph} $\G_{\W}$ for $\W \subset \V$ as the DAG with vertex set $\W$ and edge set $\mathcal{E} \cap (\W \times \W)$.

Now we define our new separation criterion: let $\A$, $\B$, $\C$ and $\D$ be disjoint sets of observed vertices.  $\A$ and $\B$ are \emph{e-separated} (\textbf{e}xtended d-separation) given $\C$ after deletion of $\D$ in $\G$, if $\A$ and $\B$ are d-separated by $\C$ in $\G_{\V \setminus \D}$.  In other words, if we remove the vertices in $\D$ from the graph, then $\A$ and $\B$ are d-separated by $\C$.

For example, in the graph in Figure \ref{fig:iv}(a), $Z$ and $Y$ are e-separated after deletion of $X$.  The following lemma gives an alternative characterization of e-separation which will prove useful.  Its proof is elementary, and omitted for brevity.

\begin{lem} \label{lem:esep}
Let $\G$ be a DAG, and let $\G^*$ be the DAG formed from $\G$ by removing all edges which are oriented away from some vertex in $\D$ (i.e.\ of the form $D \rightarrow E$ for $D \in \D$).  Then $\A$ is e-separated from $\B$ by $\C$ after deletion of $\D$ in $\G$ if and only if $\A$ is d-separated from $\B$ by $\C$ in $\G^*$.
\end{lem}

Graphs formed by removing the edges emanating from vertices form a part of Pearl's do-calculus \cite{pearl:09}.  The node-splitting method in \cite{robins:06} is also related.

Suppose now that we are interested in the detecting the presence or absence of the edge $X \rightarrow Y$ in a general graph, and in estimating the strength of the (direct) causal effect of $X$ on $Y$.  We first show that if $X$ and $Y$ are not directly confounded with each other, which is to say that they do not share a latent parent, then falsifiable constraints (such as the instrumental inequality) for the absence of the edge $X \not\rightarrow Y$ always exist.

\begin{thm} \label{thm:ineq}
Let $\G$ be a DAG, and let $\A$, $\B$, $\C$ and $\D$ be disjoint sets of observable vertices such that no vertex in $\C$ is a descendant of any in $\D$.
If $\A$ is e-separated from $\B$ by $\C$ after deletion of $\D$, then for any fixed value $\D=\d$, the conditional probabilities $p(\a, \b, \d \,|\, \c)$ must be compatible with a distribution $P^*$ in which $\A \indep \B \,|\, \C \, [P^*]$.

If in addition no vertex in $\A$ is a descendant of any element of $\D$, then the probabilities $p(\b, \d \,|\, \a, \c)$ must be compatible with a distribution $P^*$ in which $\A \indep \B \,|\, \C \, [P^*]$.
\end{thm}

%\begin{proof}
%See appendix.
%\end{proof}

\begin{cor} \label{cor:ineq}
Let $\G$ be a DAG containing observable vertices $X, Y$, which do not share a latent parent nor are joined by an edge; let $\G'$ be equal to $\G$, except that $X \rightarrow Y$ in $\G'$.  Then if the observed variables in the graphs are discrete, the model defined by the observed margin of $\G'$ is strictly larger than the one defined by $\G$.
\end{cor}

\begin{proof}
Under the conditions given, we can apply Theorem \ref{thm:ineq} to $\G$ with $\A = \{X\}$, $\B = \{Y\}$, $\C = \emptyset$ and $\D = \V \setminus (\U \cup \{X, Y\})$. 

To see that this implies a constraint, consider a distribution in which all vertices other than $X$ and $Y$ are completely independent, and $P(\D=\d) = 1-\epsilon$ for some arbitrarily small $\epsilon > 0$.  
Then $P(X, Y, \D \,|\, \C) \approx P(X, Y \,|\, \C) = P(X, Y)$, and if $X$ and $Y$ are strongly correlated, it becomes impossible to find a compatible distribution under which $X \indep Y \,|\, \C$.  However, since the only dependence is between $X$ and $Y$, such a distribution would certainly obey the global Markov property with respect to $\G'$, which contains the edge $X \rightarrow Y$.
\end{proof}

\begin{rmk}
In other words, the Corollary states there exists some non-trivial (i.e.\ falsifiable) condition on the joint distribution which must be satisfied under $\G$, but not necessarily under $\G'$.  %Note that sets of observable vertices $\C$ and $\D$ satisfying the weaker condition of the Theorem always exist, since we could just take $\C = \emptyset$, $\D = \V \setminus (\U \cup \{X,Y\})$.  
In many cases we can choose smaller sets $\D$ than the one used in the proof of Corollary \ref{cor:ineq};
the generated inequalities will tend to be more powerful if $\D$ is smaller, so certainly a minimal set should be used.

It is important to stress that this result is \emph{not} a causal one, and the constraints are merely a consequence of marginalizing distributions obeying certain conditional independence constraints.  In the next subsection, however, we will extend this method to estimate the strength of causal relationships.
\end{rmk}

In the IV graph in Figure \ref{fig:iv}(a), $Z$ is e-separated from $Y$ after deletion of $X$, giving an inequality constraint.  
%The instrumental inequality is slightly stronger than the inequality implied by the result above, because the compatibility requirement can be applied to the conditional distribution $p(y \,|\, z, x)$, rather than the joint $p(y, z \,|\, x)$.  This strengthening will hold in any graph where all the paths between $X$ and $Y$ have a pair of adjacent observed vertices; e.g.\ in the DAG in Figure \ref{fig:uc}(a), the path $X \leftarrow U_1 \rightarrow Z \leftarrow U_2 \rightarrow Y$ prevents the tighter inequality holding.  In the language of \cite{tian:02}, the stronger inequality holds if $X$ and $Y$ are in different c-components.
In general, the additional constraint implied by the Theorem may be an inequality or a conditional independence (if $\D = \emptyset$); an inequality constructed will in some cases be a weaker manifestation of a Verma constraint, or possibly some other as yet unknown form of equality constraint.  %Conditions under which independence constraints arise are determined by d-separation, while this paper makes some progress towards characterizing inequalities.  
Verma constraints are still poorly understood; see \cite{tian:02} for methods on deriving them.

\begin{rmk}
Theorem \ref{thm:ineq} can be extended to continuous state spaces without difficulty, but it is necessary for the set $\D$ to contain only discrete variables.  The IV model from Figure \ref{fig:iv}(a) with continuous $X$ is unconstrained, for example.
\end{rmk}

\subsection{Causal Bounds}

As with the IV model, we can find bounds on the average controlled direct effect due to the edge $X \rightarrow Y$ in arbitrary models, so long as $X$ and $Y$ are not directly confounded.  First we generalize the average controlled direct effect slightly to allow conditioning:
\begin{align*}
\ACDE_{X \rightarrow Y}(\d \,|\, \c) &\equiv p(y_1 \,|\, \Do(x_1, \d), \c) \\
& \qquad- p(y_1 \,|\, \Do(x_0, \d), \c).
\end{align*}
In general $\ACDE_{X \rightarrow Y}(\d \,|\, \c) \neq \ACDE_{X \rightarrow Y}(\d, \c)$, but for appropriate graphs if $\ACDE_{X \rightarrow Y}(\d \,|\, \c) \neq 0$, then $X \rightarrow Y$.

\begin{thm} \label{thm:causal}
Let $\G$ be a DAG containing the edge $X \rightarrow Y$ and observable sets of vertices $\C$, $\D$ such that no vertex in $\C$ is a  descendant of one in $\D$.  Suppose further that if the edge $X \rightarrow Y$ is removed, $X$ is e-separated from $Y$ by $\C$ after deletion of $\D$.

Let
\begin{align*}
L(x,y,\d \,|\, \c) &= \max\left\{0, \, \frac{p(x, y, \d \,|\, \c)}{p(x, \d \,|\, \c) + 1 - p(\d \,|\, \c)} \right\}\\
U(x,y,\d \,|\, c) &= \min\left\{\frac{p(x, y, \d \,|\, \c) + 1-p(\d \,|\, \c)}{p(x, \d \,|\, \c) + 1 - p(\d \,|\, \c)}, \, 1 \right\}.
\end{align*}
Then
\begin{align*}
L(x,y,\d \,|\, \c) \leq p(y \,|\, \Do(x, \d), \c) \leq U(x,y,\d \,|\, \c)
\end{align*}
and consequently for binary $X$ and $Y$,
\begin{align*}
& L(x_1,y_1,\d \,|\, \c) - U(x_0,y_1,\d \,|\, \c) \leq \ACDE_{X \rightarrow Y}(\d \,|\, \c) \\
& \qquad\qquad \leq U(x_1,y_1,\d \,|\, \c) - L(x_0,y_1,\d \,|\, \c).
\end{align*}
If in addition $X$ is not a descendant of any vertex in $\D$, these inequalities can be strengthened using 
\begin{align*}
L(x,y,\d \,|\, \c) &= p(y, \d \,|\, x, \c)\\
U(x,y,\d \,|\, \c) &= p(y, \d \,|\, x, \c) + 1 - p(\d \,|\, x, \c).
\end{align*}
\end{thm}

\begin{proof}
See appendix.
\end{proof}

\begin{rmk}
This result shows that we can always bound the effect corresponding to a directed edge, at least for some observed distributions, provided the two variables involved are not directly confounded with one another.
The bounds for the ACDE include zero if the compatibility requirement from Theorem \ref{thm:ineq} is satisfied.  If they exclude zero, then the edge $X \rightarrow Y$ must be present in the graph (given the other assumptions).
\end{rmk}

\section{Examples} \label{sec:exm}

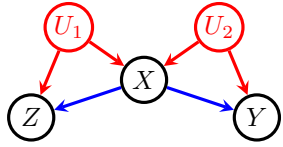
\begin{figure}
\begin{center}
\begin{tikzpicture}
[rv/.style={circle, draw, very thick, minimum size=6mm, inner sep=0.5mm}, node distance=15mm, >=stealth]
 \pgfsetarrows{latex-latex};
\begin{scope}
 \node[rv]  (1)              {$X$};
 \node[rv, left of=1, yshift=-5mm] (2) {$Z$};
 \node[rv, right of=1, yshift=-5mm] (3) {$Y$};
  \node[rv, above of=2, color=red, xshift=5mm, yshift=-3mm] (U1) {$U_1$};
  \node[rv, above of=3, color=red, xshift=-5mm, yshift=-3mm] (U2) {$U_2$};
% \node[rv, right of=2] (3) {$Y$};
 \draw[->, very thick, color=blue] (1) -- (2);
 \draw[->, very thick, color=blue] (1) -- (3);
% \draw[->, very thick, color=blue] (2) -- (3);
 \draw[->, very thick, color=red] (U1) -- (1);
 \draw[->, very thick, color=red] (U1) -- (2);
 \draw[->, very thick, color=red] (U2) -- (1);
 \draw[->, very thick, color=red] (U2) -- (3);
%\node[below of=1, yshift=3mm] {(a)};
\end{scope}
% \begin{scope}[xshift=45mm]
%  \node[rv]  (1)              {$Z$};
%  \node[rv, left of=1, yshift=-5mm] (2) {$X$};
%  \node[rv, right of=1, yshift=-5mm] (3) {$Y$};
%   \node[rv, above of=2, color=red, xshift=5mm, yshift=-3mm] (U1) {$U_1$};
%   \node[rv, above of=3, color=red, xshift=-5mm, yshift=-3mm] (U2) {$U_2$};
% % \node[rv, right of=2] (3) {$Y$};
%  \draw[->, very thick, color=blue] (1) -- (2);
%  \draw[->, very thick, color=blue] (1) -- (3);
%  \draw[->, very thick, color=blue, dashed] (2) -- (3);
%  \draw[->, very thick, color=red] (U1) -- (1);
%  \draw[->, very thick, color=red] (U1) -- (2);
%  \draw[->, very thick, color=red] (U2) -- (1);
%  \draw[->, very thick, color=red] (U2) -- (3);
%  \node[below of=1, yshift=3mm] {(b)};
% \end{scope}
\end{tikzpicture}
\end{center}
\caption{The unrelated confounding (UC) model.}%; (b) the UC model with a possible additional effect.}
\label{fig:uc}
\end{figure}

\begin{figure}
\begin{center}
\begin{tikzpicture}
[rv/.style={circle, draw, very thick, minimum size=6mm, inner sep=0.5mm}, node distance=15mm, >=stealth]
 \pgfsetarrows{latex-latex};
\begin{scope}
 \node[rv] (1) {$X$};
 \node[rv, right of=1] (2) {$Y$};
 \node[rv, above of=1] (3) {$Z$};
 \node[rv, above of=2] (4) {$W$};
 \draw[->, very thick, color=blue] (3) -- (2);
 \draw[->, very thick, color=blue] (4) -- (1);
 \draw[<->, very thick, color=red] (4) -- (2);
 \draw[<->, very thick, color=red] (1) -- (3);
 \draw[<->, very thick, color=red] (3) -- (4);
% \draw[->, very thick, color=blue] (1) -- (2);
\node[below of=1, xshift=7.5mm, yshift=3mm] {(a)};
\end{scope}
\begin{scope}[xshift=45mm]
 \node[rv] (1) {$X$};
 \node[rv, right of=1] (2) {$Y$};
 \node[rv, above of=1] (3) {$Z$};
% \node[rv, above of=2] (4) {$W$};
 \draw[->, very thick, color=blue] (3) -- (2);
% \draw[->, very thick, color=blue] (4) -- (1);
% \draw[<->, very thick, color=red] (4) -- (2);
 \draw[<->, very thick, color=red] (1) -- (3);
% \draw[<->, very thick, color=red] (3) -- (4);
% \draw[->, very thick, color=blue] (1) -- (2);
 \node[below of=1, xshift=7.5mm, yshift=3mm] {(b)};
\end{scope}
\end{tikzpicture}
\end{center}
\caption{(a) A DAG with three independent unobserved variables; we have avoided explicitly drawing a vertex for each of the three unobserved variables, and instead use a bidirected ($\leftrightarrow$) edge to indicate its two (observed) children. (b) The same graph after deletion of $W$.}
\label{fig:gadget}
\end{figure}
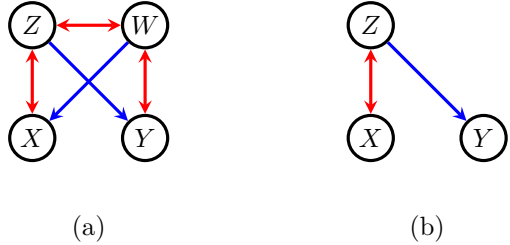

The graph in Figure \ref{fig:uc}, which we refer to as the \emph{unrelated confounding} (UC) model, has no edge between $Z$ and $Y$, and nor are these two variables directly confounded.
Theorem \ref{thm:ineq} and Corollary \ref{cor:ineq} therefore tell us that in the discrete case, the joint distribution of $(X,Y,Z)$ is restricted, and in particular that for each $\xi$, the joint probabilities $p(\xi, y, z)$ must be compatible with a distribution in which $Z \indep Y$.  Let $p_{ijk} \equiv p(x_i, y_j, z_k)$; in the binary case, given $p_{000}, p_{010}, p_{001}, p_{011}$, we need to find non-negative $p^*_{100}, p^*_{110}, p^*_{101}, p^*_{111}$ such that
\begin{align*}
(p_{000} + p^*_{100})(p_{011} + p^*_{111}) = (p_{010} + p^*_{110})(p_{001} + p^*_{101})
\end{align*}
and $\sum_{jk} p_{0jk} + \sum_{jk} p^*_{1jk} = 1$.  This will not be possible if, for example, $p_{000}$ and $p_{011}$ are both large; that is, we cannot have both $P(X=\xi)$ be large \emph{and} $Z$ and $Y$ strongly correlated conditional on $X=\xi$.
%One consequence of this condition is that
%\begin{align*}
%(1-p_{000})^2 + (1-p_{011})^2 \geq 1.
%\end{align*}
Unlike in the IV model we cannot apply the stronger condition of Theorem \ref{thm:ineq}, because $Z$ is a descendant of $X$.  We remark that (observationally) the UC model strictly contains the IV model in Figure \ref{fig:iv}(a).
Note that a linear programming approach to finding constraints on this graph is not possible, so the constructive nature of the proof of Theorem \ref{thm:ineq} is crucial in determining how we can test this model.

The graph in Figure \ref{fig:gadget}(a) is constrained in the discrete case because there is no edge between $X$ and $Y$.  Specifically $X$ is e-separated from $Y$ given $W$ after deletion of $Z$, and also given $Z$ after deletion of $W$ (the latter being illustrated in Figure \ref{fig:gadget}(b)).  Note that $X$ is a descendant of $W$, but not of $Z$, so the bounds given by Theorem \ref{thm:causal} are not symmetric in the two cases.  For example:
\begin{align*}
p(y \,|\, \Do(x,w), z) \leq p(y, w \,|\, x, z) + 1 - p(w \,|\, x, z)\\
p(y \,|\, \Do(x,z), w) \leq \frac{p(x, y, z \,|\, w) + 1 - p(z \,|\, w)}{p(x, z \,|\, w) + 1 - p(z \,|\, w)}.
\end{align*}
The first bound is likely to be stronger, though this will not hold in all cases.

\section{Discussion} \label{sec:discuss}

We have presented a graphical approach to finding inequality constraints in distributions corresponding to marginalized DAGs, based on the e-separation criterion.  It can be shown that the bounds derived from the algorithm of \cite{kang:06} also imply the causal constraints given in Theorem \ref{thm:causal}, however that approach involves listing exponentially many inequalities and then using Fourier-Motzkin elimination to derive bounds.  For even modestly sized graphs this becomes infeasible because Fourier-Motzkin is doubly-exponential in the number of variables in the elimination.

The advantage of the results given above is that they are `off the shelf', in the sense that we need only check the conditions of the Theorems and then apply the results.
Exhaustively searching possible sets $\C$ and $\D$ would be computationally intensive, but in many cases it is likely that good heuristics could be obtained for their selection.
This could be highly advantageous in systems with large numbers of variables, especially during computationally intensive model search procedures.
A further benefit of the e-separation criterion is that it is much easier and more intuitive for a human user to apply than using the algorithm of \cite{kang:06}.

The bounds derived from Theorem \ref{thm:ineq} are known not to be tight in some cases, including the IV model when the instrument takes three or more states.  However finding constraints from marginalized models is computationally intensive, even if the inequalities are linear, so a fast method for finding a subset of conditions may be very useful in practice.

\appendix

\section{Proofs}

\begin{proof}[Proof of Theorem \ref{thm:ineq}]
By the global Markov property for DAGs, the joint distribution $P$ over the observed variables takes the form (\ref{eqn:factor}).
Now, for each factor $f(V \,|\, \pa_{\G}(V))$, construct a new conditional density $f^*(V \,|\, \pa^*(V))$ where $\pa^*(V) = \pa_{\G}(V) \setminus \D$, by fixing any element of $\D \cap \pa_{\G}(V)$ to the value specified by $\D=\d$.  Note we only fix elements in the conditioning set, so $f^*$ is still a valid conditional density.

Then the joint distribution $P^*$ given by
\begin{align*}
\int_{U \in \U}  \prod_{V \in \V} f^*(V \,|\, \pa^*(V)) \, dU
\end{align*}
factorizes according to the DAG $\G^*$ formed by removing any edges in $\G$ which originate in $\D$ (i.e.\ the non-arrowhead end is incident to a vertex in $\D$).  By Lemma \ref{lem:esep}, $\A$ and $\B$ are d-separated by $\C$ in $\G^*$, and therefore the global Markov property for DAGs says that $\A \indep \B \,|\, \C \, [P^*]$.  Further, $P(\a, \b, \c, \d) = P^*(\a, \b, \c, \d)$ for the fixed $\D = \d$ and any $\a,\b,\c$, and $P^*(\c) = P(\c)$ because the distribution of vertices ordered before $\D$ will be unchanged.  This gives the compatibility condition.  If $\A$ is also ordered before $\D$ then $P^*(\a, \c) = P(\a ,\c)$, giving the stronger condition.
\end{proof}

\begin{proof}[Proof of Theorem \ref{thm:causal}]
For simplicity we will assume $\C = \emptyset$, but the extension to the general case is easy.
Let $P^*$ be the distribution formed by fixing $\D=\d$ in conditioning sets in the factorization of $P$, as in the proof of Theorem \ref{thm:ineq}.  Then
\begin{align*}
p(y \,|\, \Do(x, \d)) &= p^*(y \,|\, x) = \frac{p^*(y, x)}{p^*(x)}\\
&= \frac{p(y, x, \d) + \sum_{\d'\neq \d} p^*(y, x, \d')}{p(x, \d) + \sum_{\d'\neq \d} p^*(x, \d')}
\end{align*}
Clearly $\sum_{\d'\neq \d} p^*(y, x, \d') \leq \sum_{\d'\neq \d} p^*(x, \d') \leq 1 - p(\d)$; the expression is maximized by both these sums taking their largest possible values, and minimized when the first is zero and the second is $1 - p(\d)$.
%\begin{align*}
%\frac{p(y, x, \d)}{p(x, \d) + q(\d)} \leq p(y \,|\, x, \Do(\d)) &\leq \frac{p(y, x, \d) + q(\d)}{p(x, \d) + q(\d)}
%\end{align*}
%where $q(\d) \equiv 1-p(\d)$.  
This gives the main result.  %Using the definition of the average causal direct effect gives the main result.
If $X$ is not a descendant of $\D$ we have $p^*(x) = p(x)$, and arrive at the tighter bounds by a similar analysis.
%\begin{align*}
%\frac{p(y, x=1, d)}{p(x=1, d) + 1 - p(d)} &- \frac{p(y, x=0, d) + 1 - p(d)}{p(x=0, d) + 1 - p(d)} \leq \\
%\ACDE_{X \rightarrow Y}(d) &\leq \frac{p(y, x=1, d) + 1 - p(d)}{p(x=1, d) + 1 - p(d)} - \frac{p(y, x=0, d)}{p(x=0, d) + 1 - p(d)}
%\end{align*}
%\begin{align*}
%\frac{p(y_1, x_1, d)}{p(x_1, d) + 1 - p(d)} &- \frac{p(y_1, x_0, d) + 1 - p(d)}{p(x_0, d) + 1 - p(d)} \leq \\
%\ACDE_{X \rightarrow Y}(d) &\leq \frac{p(y, x_1, d) + 1 - p(d)}{p(x_1, d) + 1 - p(d)} - \frac{p(y, x_0, d)}{p(x_0, d) + 1 - p(d)}.
%\end{align*}
\end{proof}

%---------------%
\bibliographystyle{plainnat}
\bibliography{../mybib}

\end{document}